\documentclass[11pt,twoside,reqno]{amsart} 

\usepackage[utf8]{inputenc}
\usepackage[T1]{fontenc}

\usepackage{enumitem}
\usepackage[bookmarks = true, pagebackref]{hyperref}
\hypersetup{  
	colorlinks=true,
	linkcolor=blue,
	citecolor=green,
	urlcolor=magenta,
}
\usepackage[all]{xy}
\usepackage{amsmath, amsthm, amssymb, bbm, mathrsfs,mathtools}
\usepackage[alphabetic]{amsrefs} 

\usepackage{varwidth}

\newtheorem{theorem}{Theorem}[section]
\newtheorem*{theorem*}{Theorem}
\newtheorem{lemma}[theorem]{Lemma}
\newtheorem{proposition}[theorem]{Proposition}
\newtheorem{corollary}[theorem]{Corollary}

\theoremstyle{definition}
\newtheorem{remark}[theorem]{Remark}

\newtheorem{definition}[theorem]{Definition}
\newtheorem{example}[theorem]{Example}

\numberwithin{equation}{section}

\makeatletter

\def\@buildmath#1{%
  \expandafter\def\csname mb#1\endcsname{\mathbb{#1}}%
  \expandafter\def\csname mc#1\endcsname{\mathcal{#1}}%
  \expandafter\def\csname mf#1\endcsname{\mathfrak{#1}}%
  \expandafter\def\csname rm#1\endcsname{\mathrm{#1}}%
  \expandafter\def\csname bf#1\endcsname{\mathbf{#1}}%
  \expandafter\def\csname sf#1\endcsname{\mathsf{#1}}%
}

\def\@buildmathletters#1{%
  \ifx#1\relax
    \expandafter\@gobble
  \else
    \@buildmath{#1}%
    \expandafter\@buildmathletters
  \fi
}

\@buildmathletters ABCDEFGHIJKLMNOPQRSTUVWXYZabcdefghijklmnopqrstuvwxyz\relax

\makeatother

\newcommand{\Hom}{\operatorname{Hom}}

\newcommand{\End}{\operatorname{End}}

\newcommand{\Lie}{\operatorname{Lie}}
\newcommand{\Spec}{\operatorname{Spec}}

\newcommand{\GL}{\operatorname{GL}}

\newcommand{\GU}{\operatorname{GU}}

\newcommand{\id}{\mathrm{id}}

\newcommand{\Stab}{\operatorname{Stab}}

\newcommand{\loc}{\mathrm{loc}}

\newcommand{\tr}{\mathrm{tr}}

\newcommand{\Sh}{\mathrm{Sh}}

\newcommand{\Res}{\operatorname{Res}}

\newcommand{\Sym}{\operatorname{Sym}}
\newcommand{\IGr}{\mathrm{IGr}}
\newcommand{\Gr}{\mathrm{Gr}}
\newcommand{\Fl}{\mathrm{Fl}}
\newcommand{\adj}{\mathrm{adj}}
\newcommand{\im}{\operatorname{im}}
\newcommand{\ch}{\operatorname{ch}}
\newcommand{\coker}{\operatorname{coker}}

\newcommand{\wt}{\widetilde}

\newcommand{\lr}{\longrightarrow}
\newcommand{\ov}{\overline}

\title{Pappas's conjecture on ramified unitary local models}
	
\author{Yu LUO}
\address{Massachusetts Institute of Technology, Department of Mathematics, 77 Massachusetts Avenue, Cambridge, MA 02139, USA}
\email{yuluo@mit.edu}
	
\date{\today}
	
\begin{document}
\begin{abstract}
We prove that the wedge local model for a ramified unitary group of signature $(r,s)$ ($r>s$) at self-dual level is flat when $p\geq 2s-1$, confirming Pappas's conjecture in this case. The proof first reduces the problem to determining the defining equations of a nilpotent orbit closure of self-adjoint matrices, and then to computing the global sections of the corresponding Kempf resolution. The computation is carried out by reducing to the type $A$ case using the Koszul complex.
\end{abstract}

\maketitle
\tableofcontents
\section{Introduction}
\subsection{Background}
Shimura varieties, developed by Deligne \cite{Deligne71,Deligne79} in the 1970s, generalize classical objects such as modular curves, moduli spaces of principally polarized abelian varieties and Hilbert modular varieties. 
Roughly speaking, they are algebraic varieties that carry both Hecke symmetries and Galois symmetries, making them interesting and important objects in arithmetic geometry. In many cases, their arithmetic properties are encoded in the reduction to positive characteristic $p>0$ and have contributed to spectacular developments in number theory and arithmetic geometry in the past decades.

Local models are flat projective schemes over complete discrete valuation rings of characteristic $(0,p)$ that are designed to model the singularities in the reduction modulo $p$ of Shimura varieties with (quasi-)parahoric level structure. 
The theory was first formalized for Shimura varieties of PEL type in the book of Rapoport--Zink \cite{RZ96}. Since then, much research has been devoted to local models and their relation to Shimura varieties. A breakthrough is due to Pappas and Zhu \cite{PZ13,Zhu_coherent}, who gave a purely group-theoretic construction of flat local models. Later, Scholze and Weinstein \cite{SW20} proposed a more canonical approach using $p$-adic geometry, which was recently completed by \cite{AGLR,GL_tubular,FHLR,CL26}.

However, as mentioned in \cite{AGLR}, the group-theoretic approach does not provide a moduli interpretation of local models, which remains an interesting and difficult problem that may not have a uniform solution. Explicit moduli descriptions of local models, and consequently of the corresponding integral models of Shimura varieties, are important for arithmetic intersection problems, including Kudla--Rapoport conjecture \cite{KR11,KR14,LZ22,LL22,Yao_KR,HLSY,Luo_KR}, arithmetic inner product formula \cite{LL21,LL22,Chi_AIPF}, arithmetic fundamental lemma \cite{WZhang_AFL,WZhang_AFL-proof,Mihatsch-Zhang}, and arithmetic Gan--Gross--Prasad conjecture \cite{RSZ_regular,DZ}. In all of these problems, unitary Shimura varieties, unitary local models, and their moduli descriptions play an important role.

Let $F/F_0$ be a quadratic extension and let $V$ be a hermitian space. For any pair $(r,s)$ with $r+s=\dim V$, one can consider unitary Shimura varieties and local models of signature $(r,s)$. When the extension $F/F_0$ is split or inert, the moduli problem proposed in \cite{RZ96} is known to be flat by \cite{Gortz01}. However, when $F/F_0$ is ramified, finding a moduli description becomes a much more subtle problem. The case of signature $(r,s)=(n-1,1)$ with self-dual level structure was first considered by Pappas \cite{Pappas2000}. 
Since then, subsequent work includes \cite{PR2008,PR09,Arzdorf,Richarz13,Smithling_odd,Smithling_even,RSZ_regular,Smithgling15}. This case was recently settled by the author in \cite{Luo_LM}. For signatures beyond $(n-1,1)$, recent work on integral models includes \cite{Kramer,Zachos24,ZZI,ZZII,BH25,BZZ}. However, little was known about moduli interpretations of the \emph{canonical} integral models in these signatures, except for the case of signature $(4,2)$ considered in \cite{Trentin}.
The main goal of this paper is to take a first step toward a moduli interpretation of the integral models of ramified unitary Shimura varieties for general signature $(r,s)$.

\subsection{Main result}
Let $F/F_0$ be a ramified quadratic extension of $p$-adic fields with $p>2$, and let $\Lambda$ be a self-dual lattice in an $n$-dimensional hermitian space.

For integers $(r,s)$ with $r+s=n$, let $\mbM_{(r,s)}^\wedge$ be the wedge local model of signature $(r,s)$, we refer the reader to Definition \ref{def:wedgedLM} for a precise definition. Roughly speaking, it parameterizes filtrations of the hermitian lattice that are compatible with the $O_F$-action and are modeled on the Hodge filtrations of abelian schemes. 
The local model is the Zariski closure of the generic fiber inside the wedge local model:
$$
\mbM^\wedge_{(r,s),\eta}\hookrightarrow\mbM_{(r,s)}^{\loc}\hookrightarrow \mbM_{(r,s)}^{\wedge}.
$$
The wedge moduli functor gives the moduli interpretation precisely when the closed immersion $\mbM_{(r,s)}^\loc\hookrightarrow \mbM_{(r,s)}^{\wedge}$ is an isomorphism. The main result of the paper is the following:
\begin{theorem}(Theorem \ref{thm:Pappas-rs})
Suppose $p\neq 2$. Let $(r,s)$ be such that $r+s=n$ and $r>s$, and assume $p\geq 2s-1$. For every ramified quadratic extension $F/F_0$ and every self-dual lattice $\Lambda$ as above, we have
$$
\mbM_{(r,s)}^\loc=\mbM_{(r,s)}^{\wedge}.
$$
In particular, the wedge local model is flat and has a geometrically reduced special fiber. This proves a conjecture of Pappas, stated implicitly in \cite[\S 4.16]{Pappas2000}.
\end{theorem}
Through the local model diagram, this equality gives a flat canonical PEL moduli problem for unitary Shimura varieties of signature $(r,s)$ with principal polarization, see \S \ref{sec:moduliforSV} for more discussion. Here canonical is in the sense of \cite{SW20,PR-shtuka}. We note that a splitting model variant was also proposed in \cite{Kramer} and studied in \cite{BH25} for any signature. 
Our result is the first moduli description of the canonical integral model in the ramified unitary case beyond signature $(n-1,1)$ for arbitrarily large dimension.

The key algebraic result identifies, in an affine neighborhood of the worst point, the defining equations of the special fiber of the local model, using nilpotent orbit closure.

To state it, let $k$ be an algebraically closed field and let $(\bfF,q)$ now be a nondegenerate quadratic space over $k$. We identify $\bfN=\Sym^2\bfF$ with the space of $q$-self-adjoint endomorphisms of $\bfF$. In an orthonormal basis, these are symmetric matrices.
Let $N$ be the corresponding affine variety and let $N_{u}$ be the reduced closed subvariety of matrices defined by $X^2=0$ and $\mathrm{rank}(X)\leq s$. 

\begin{theorem}(Theorem \ref{thm:main-equation})
Suppose $p=0$ or $p\geq 2s-1$ ($p\neq 2$). Assume $n\geq 2s+1$. The defining ideal of $N_{u}$ in $k[N]$ is 
$$
I(N_{u})=\bigl(
\tr(X),(X^2)_{ij},I_{s+1}(X)
\bigr),
$$
where $I_{s+1}(X)$ is the ideal generated by the $(s+1)\times (s+1)$ minors of $X$.
\end{theorem}

\subsection{Proof idea}
An important observation made in \cite[\S 5.4]{PR09} is that $N_{u}$ is in fact a nilpotent orbit closure of self-adjoint matrices. The analogous problem in type $A$ was studied by \cite{Weyman_89}, and the method was further developed in \cite{Weyman_03}. 

Let $\IGr(s,\bfF)$ be the isotropic Grassmannian. Let $\mcR\subset\mcF$ be the tautological subbundle. We consider 
$$
Y_{u}\subset N\times \IGr(s,\bfF)\to \IGr(s,\bfF).
$$
Here $Y_{u}$ is the total space of $\Sym^2\mcR\subset\Sym^2\mcF$. The geometric points of $Y_{u}$ can be described as
$$
Y_{u}(k)=\{(X,R)\in N\times \IGr(s,\bfF)\mid X\bfF\subseteq R\}.
$$
In particular, every matrix $X$ occurring here has rank $\leq s$, and the projection to $N$ is a Kempf type resolution \cite{Kempf} of $N_u$
\begin{equation*}
\begin{aligned}
\xymatrix{
N_{u}&Y_{u}\ar[r]\ar[l]_-{\pi_{u}}&\IGr(s,\bfF)
}
\end{aligned}
\end{equation*}
and the ring of regular section of $N_u$ is identical to that of $Y_u$.

Since $Y_{u}$ is the total space of a vector bundle over the isotropic Grassmannian, the structure sheaf (in fact, cohomology) of $Y_{u}$ can be studied using geometric representation theory. This gives a representation-theoretic description of the ring of functions on $N_{u}$. However, extracting generators for the defining ideal from this description remains difficult, even when the ideal is known explicitly as a representation.

To proceed, we reduce the problem to the type $A$ setting. Let $\Gr(s,\bfF)$ be the usual Grassmannian. Since all bundles we consider above are pulled back from $\Gr(s,\bfF)$, our resolution fits into the following diagram:
\begin{equation*}
\begin{aligned}
\xymatrix{
Y_{u}\ar[dd]\ar@{^(->}[rr]\ar[dr]&\ar@{}[dr]|{\Box}&N\times \IGr(s,\bfF)\ar[r]\ar@{^(->}[d]\ar@{}[dr]|{\Box}&\IGr(s,\bfF)\ar@{^(->}[d]\\
&Y_s\ar@{^(->}[r]\ar[d]&N\times\Gr(s,\bfF)\ar[r]\ar[d]&\Gr(s,\bfF)\\
N_{u}\ar@{^(->}[r]&N_s\ar@{^(->}[r]&N.
}
\end{aligned}
\end{equation*}
The subvariety $Y_s$ is the total space of $\Sym^2\mcR$ (without the isotropic requirement), and $N_s$ is nothing but the symmetric determinantal variety of rank $\leq s$. Its geometric structure is well understood.

The closed immersion $\IGr(s,\bfF)\subset\Gr(s,\bfF)$ is a local complete intersection. We use this to relate cohomology on $\IGr(s,\bfF)$ to cohomology on $\Gr(s,\bfF)$ using Koszul resolution. This reduces the problem to representation theoretic calculations $\GL_s(k)$.

The reduction procedure works perfectly when $p=0$. In positive characteristic, we show that the terms of the Koszul complex admit good filtrations, which requires some care with modular representation theory.

Let us finish the section with several remarks.
\begin{remark}\label{rmk:intro}
\begin{enumerate}
\item The framework should also apply to signature $(r,r)$, although that case seems to require a different strategy.
\item In \cite[\S 5.4]{PR09}, Pappas and Rapoport also consider a nilpotent orbit closure at the $\pi$-modular level. We expect that a similar approach applies.
It would be interesting to generalize this from self-dual level structure to maximal parahoric level, as in the case of signature $(n-1,1)$ in \cite{Luo_LM}. 
\item The Kempf resolution $Y_u\to N_{u}$ considered in this paper coincides over the special fiber with (the proper transform of) the Kr\"amer resolution \cite{Kramer}. This suggests that Kr\"amer type resolutions have some canonical geometric meaning, even though they do not have a canonical definition. It would be interesting to compare this resolution with those considered in \cite{HLS_regular,ZZII,LRZ25}.
\end{enumerate}
\end{remark}

\begin{remark}
Since the initial draft of this manuscript, we have used AI to find a different strategy for proving Pappas's conjecture using invariant theory. The resulting proof appears to be substantially shorter, although it is also technically more involved. We intend to explore this method in the future.
\end{remark}

\subsection{Acknowledgments}
The idea of using representation theory to study local models was suggested by M. Rapoport and G. Pappas in 2023 (for different purposes). The idea for this paper arose during the banquet of ICM 2026. The author's travel was supported by the National Science Foundation under Grant No. 2529637. The author would like to thank M. Rapoport for helpful discussions and comments, and I. Zachos for comments on an early draft.

\subsection{AI usage}
The author used ChatGPT and Claude for general research assistance (e.g. literature searches and exploring proof strategies). 
In particular, initial proof strategies for Lemma \ref{lem:good-fil-surj} and Proposition \ref{prop:char0-calculation} were suggested by ChatGPT Astra 6.0 and Claude Opus 5.5, these strategies generalize a more naive argument that works in dimension two.

The author used Claude Opus 5.5 to review the paper and polish the writing. During this review, Claude Opus 5.5 also improved the bound in Proposition \ref{prop:wedge-good} from $p>s(s+1)/2$ to $p\geq 2s-1$, and suggested a simplification that avoids \cite[Theorem 5.1.2]{Weyman_03}, which the first draft used. The author takes full responsibility for the content of this article.

\section{Vector bundles on Grassmannians}\label{sec:weyman}
\subsection{Schur functors}
We recall the construction of Schur functors and their properties. Most notation conventions follow \cite{SSam16}.

Let $X$ be a scheme and let $\mcE$ be a vector bundle of rank $s$. 
For a partition $\lambda$ with transpose $\lambda'$, we define the \emph{Schur functor} $S_\lambda(\mcE)$ to be the image of the natural map
$$
\bigwedge^{\lambda'_1}\mcE\otimes \cdots \otimes \bigwedge^{\lambda'_{\lambda_1}}\mcE\lr
\Sym^{\lambda_1}\mcE\otimes\cdots\otimes \Sym^{\lambda_t}\mcE.
$$
This is $L_{\lambda'}(\mcE)$ in \cite[\S 2.1]{Weyman_03}. We choose this notation for better compatibility with representations.

For any dominant weight $\lambda=(\lambda_1\geq \cdots\geq \lambda_s)\in\mbZ^s$, we define 
$$
S_{\lambda}(\mcE):=S_{(\lambda_1-\lambda_s,\cdots,\lambda_{s-1}-\lambda_s,0)}(\mcE)\otimes \det(\mcE)^{\otimes \lambda_s}.
$$
These functors are well-defined and functorial, see \cite{ABW-schur}.

We recall some properties of Schur functors.
For any vector bundle $\mcV$, we write 
$$
\mcV\approx \bigoplus \mcV_i
$$
if $\mcV$ admits a filtration with associated graded part isomorphic to the right hand side of the formula.

\begin{theorem}\label{thm:schur-decomp}
\begin{enumerate}
\item For any dominant weights $\lambda,\mu\in \mbZ^s$, there is a filtration such that
$$
S_{\lambda}(\mcE)\otimes S_\mu(\mcE)\approx  
\bigoplus_{\substack{|\nu|=|\lambda|+|\mu|\\ \nu_s\geq \lambda_s+\mu_s}} S_\nu(\mcE)^{\oplus c(\lambda,\mu;\nu)}
$$
where $\nu$ runs over the dominant weights. See \cite[Theorem 2.3.4]{Weyman_03} and \cite{Boffi-LR} for a formula for $c(\lambda,\mu;\nu)$. The filtration is universal in the sense that it depends only on $\lambda,\mu$.
\item For any $m\geq 0$, there is a filtration
$$
\Sym^m(\Sym^2\mcE)\approx \bigoplus_{\substack{|\lambda|=2m,\\
\lambda_i\text{ even for all }i}}S_{\lambda}(\mcE).
$$
The filtration is universal in the sense that it depends only on $m$.
\end{enumerate}
\end{theorem}
\begin{proof}
Part (1) is proved in \cite{Boffi-LR}, see also \cite[\S 2.3]{Weyman_03}. Part (2) is \cite[\S 3]{Boffi-Plethysms}.
\end{proof}

We remark that $\bigwedge^m(\Sym^2(\mcE))$ does not admit a universal filtration, see \cite[\S 4]{Boffi-Plethysms}.

Schur functors also extend to complexes and to derived categories. We conclude with the following example.
\begin{theorem}\label{thm:Schur-complex}
Let 
$$
0\lr \mcV_1\lr \mcV_2\lr \mcV_3\lr 0
$$ 
be an exact sequence of vector bundles. We have a long exact sequence
$$
0\lr \bigwedge^2 \mcV_1\lr \mcV_1\otimes \mcV_2\lr \Sym^2 \mcV_2\lr \Sym^2 \mcV_3\lr 0.
$$

\end{theorem}
\begin{proof}
This is a special case of the construction of Schur functors for complexes \cite[\S 2.4]{Weyman_03}.
We describe the construction explicitly in this case. Let $\Phi:\Phi_1\to \Phi_0$ be a two-term complex. By 
the discussion before \cite[Proposition 2.4.1]{Weyman_03}, we see that
$$
\Sym^d\Phi:\bigwedge^d\Phi_1\lr \bigwedge^{d-1}\Phi_1\otimes\Phi_0\lr \cdots\lr \Phi_1\otimes\Sym^{d-1}\Phi_0\lr \Sym^d\Phi_0\lr 0.
$$
Now consider the complex $\mcV_1\overset{\rho}{\lr} \mcV_2$, where $\mcV_1$ is in degree $1$ and $\mcV_2$ is in degree $0$, then
$$
\Sym^2(\mcV_1\lr \mcV_2)=\Bigl[
0\lr \bigwedge^2 \mcV_1\lr \mcV_1\otimes \mcV_2\lr \Sym^2 \mcV_2
\Bigr]
$$
and by \cite[Theorem 2.4.10]{Weyman_03}, it is (locally) quasi-isomorphic to $\Sym^2 \mcV_3$, concentrated in degree $0$.
\end{proof}

\subsection{Modular representations}
In this subsection, we review some basic facts about modular representations. We refer the reader to \cite[\S 4]{AJ84} and \cite[\S 2.4]{Lorincz23} for a good summary.

Let $k$ be an algebraically closed field of characteristic $p$. Let $G$ be a (connected) reductive group over $k$, with $T\subset B\subset G$ a torus and Borel for $G$. For any character $\lambda$ of $T$, we have the induced representation
$$
\nabla_\lambda=H^0(G/B,\mcO(\lambda)).
$$
When $G=\GL(\bfV)$, by \cite[Theorem 4.1.10]{Weyman_03}, we have $\nabla_\lambda=S_\lambda(\bfV)$. In the remaining part of the paper, we will focus on this situation.

Let $\bfV$ be a $G$-module.
Following \cite{Donkin85}, an ascending exhaustive filtration
$$
0=\bfV_0\subset \bfV_1\subset \bfV_2\subset \cdots
$$
of $G$-submodules of $\bfV$ is called a \emph{good filtration} of $\bfV$ if each $\bfV_i/\bfV_{i-1}$ is isomorphic to $\nabla_\lambda$ for some $\lambda$. 

If one picks a total ordering $0=\lambda_0,\lambda_1,\cdots$ of the dominant cocharacters, compatible with their dominance order, one can choose $\bfV_i\subset \bfV$ to be the maximal subrepresentation admitting only weights $\lambda_j$ with $j\leq i$. In that case, $(\bfV_i/\bfV_{i-1})^\vee$ is generated by its highest weight space $\bfW_i^*$, and by adjunction there is a map
$$
\bfV_i/\bfV_{i-1}\lr \bfW_i\otimes \nabla_{\lambda_i};
$$
then $\bfV$ admits a good filtration if and only if all of these maps are isomorphisms.

We summarize the properties we will use.
\begin{lemma}\label{lem:goodfil}
Let $\bfV$ and $\bfW$ be finite dimensional $G$-modules.
\begin{enumerate}
\item Consider the long exact sequence
$$
0\lr \bfV_n\lr \bfV_{n-1}\lr \cdots\lr \bfV_1\lr \bfV\lr 0
$$
of finite dimensional $G$-modules. If all $\bfV_i$ have good filtrations, then so does $\bfV$.
\item If $\bfV$ and $\bfW$ admit good filtrations, then $\bfV\otimes \bfW$ has a good filtration.
\item If $\langle \chi+\rho,\alpha^\vee\rangle\leq p$ for all weights $\chi$ of $\bfV$ and all positive roots $\alpha$, then $\bfV$ has a good filtration.
\end{enumerate}
\end{lemma}
\begin{proof}
See \cite[\S 4.1]{AJ84}.
\end{proof}

The result we need is the following:
\begin{proposition}\label{prop:wedge-good}
Suppose $p=0$, or $p\geq 2s-1$ ($p\neq 2$). Let $\bfR$ be a vector space of dimension $s$, and let $m:=\dim \Sym^2 (\bfR)=s(s+1)/2$. Then:
\begin{enumerate}
\item For any $0\leq j\leq m$, the module $\bigwedge^j(\Sym^2 (\bfR))$ has a good filtration, with associated graded part 
$$
\bigwedge^{j}(\Sym^2(\bfR))\approx \bigoplus_{\lambda\in Q_{1}(2j)}S_{\lambda}(\bfR)
$$
where $Q_{1}(2j)$ is the set of partitions (in hook notation) $\lambda=(a_1+1,\cdots,a_k+1|a_1,\cdots,a_k)$ is a partition of $2j$, such that $a_1>\cdots>a_k>0$, see \cite[\S 1.1.2]{Weyman_03}.

\item For any $0\leq j\leq m$ and $d\geq 0$, the module 
$$
\bigwedge^j(\Sym^2 (\bfR))\otimes \Sym^d(\Sym^2 (\bfR^\vee))
$$
has a good filtration, whose factors $S_{\gamma}(\bfR^\vee)$ satisfy $\gamma_s\geq -(s+1)$.
\end{enumerate}
\end{proposition}
\begin{proof}
When $p=0$, this is standard, in the remaining proof, we assume $p\geq 2s-1$. We use Lemma \ref{lem:goodfil}(3). Let $\alpha^\vee=\epsilon_a-\epsilon_b$ be a positive coroot with $a<b$. For any weight $\chi$ of $\bigwedge^j(\Sym^2(\bfR))$ and any $1\leq i\leq s$, we have $0\leq \chi_i\leq s+1$, with the maximal one given by
$$
e_ie_1\wedge e_ie_2\wedge\cdots\wedge e_ie_s.
$$
We claim that $\chi_a-\chi_b\leq s$: if $\chi_a=s+1$, then $\chi_b\geq 1$ for all $b\neq a$, hence $\chi_a-\chi_b\leq  s$, the other cases follow trivially.

Hence 
$$
\langle \chi+\rho,\alpha^\vee\rangle=\chi_a-\chi_b+(b-a)\leq  s+(s-1)=2s-1\leq  p.
$$
Hence $\bigwedge^j(\Sym^2 (\bfR))$ has a good filtration. Once we have a good filtration, the associated graded part of the filtration is uniquely determined by the formal character (which is defined for any finite dimensional $G$-module). The assertion then follows from the corresponding Plethysms identity.

For part (2), we have a good filtration by Lemma \ref{lem:goodfil}(2), it remains to bound the weights.

Let $S_{\lambda}(\bfR)$ be a factor of $\bigwedge^j(\Sym^2(\bfR))$, then $\lambda=(a_1+1,\cdots,a_k+1|a_1,\cdots,a_k)$. In particular, we have $\lambda_1=a_1+1$ and $\lambda$ has $a_1=\lambda_1-1$ rows. Since $\bfR$ has rank $s$, we have $\lambda_1-1\leq s$, hence $\lambda_1\leq s+1$.
On the other hand, all $\mu\in\mbZ^s$ such that $S_\mu (\bfR^\vee)$ contributes to $\Sym^d(\Sym^2 (\bfR^\vee))$ have $\mu_i\geq 0$. By Theorem \ref{thm:schur-decomp}(1), all $\nu=(\nu_1,\cdots,\nu_s)$ such that $S_{\nu}(\bfR^\vee)$ appears in the tensor product
$$
S_{(-\lambda_s,\cdots,-\lambda_1)}(\bfR^\vee)\otimes S_{\mu}(\bfR^\vee)
$$
must have $\nu_s\geq -\lambda_1+\mu_s\geq -(s+1)$.
\end{proof}

\subsection{Vector bundles on Grassmannians}
In this subsection, we discuss vector bundles on Grassmannians. Let $\bfF$ be a vector space of dimension $n$ over $k$. Consider the Grassmannian $\Gr=\Gr(s,\bfF)$ with tautological bundle $\mcR\subseteq\mcF:=\bfF\otimes_{k}\mcO_{\Gr}$. We want to consider Schur functors and their filtrations on $\Gr$. Note that there is no notion of good filtration for families.

\begin{proposition}\label{prop:exact-fun}
Let $G$ be an affine algebraic group and $H\subset G$ a closed subgroup. For each finite dimensional rational $H$-module $\bfW$, consider the sheaf of $\mcO_{G/H}$-modules $\mcL_{G/H}(\bfW)$ which sends an open subset $U\subset G/H$ to
$$
\mcL_{G/H}(\bfW)(U)=\{\text{sections of }G\times^{H}\bfW\lr G/H\text{ above }U\}.
$$
The functor $\mcL_{G/H}(-)$ induces an equivalence of categories between finite dimensional rational $H$-modules and $G$-equivariant algebraic vector bundles on $G/H$. This functor is exact and commutes with tensor powers, duals, symmetric powers, divided powers, exterior powers, and Frobenius twists. In particular, it commutes with Schur functors.
\end{proposition}
\begin{proof}
This is \cite[Proposition 3.2]{CFP16}, where the last part follows from the construction of the Schur functor.
\end{proof}

\begin{example}\label{exmp:grassmannian}
Let $\bfF$ be a vector space of dimension $n$ and let $\bfR\subset \bfF$ be a subspace of dimension $s$. 
Let $P_s\subset \GL(\bfF)$ be the parabolic stabilizing $\bfR$, with Levi factor $\GL(\bfR)\times \GL(\bfF/\bfR)$. 
Via the projection $P_s\to \GL(\bfR)$, the standard representation of $\GL(\bfR)$ becomes a $P_s$-module, and the tautological bundle is $\mcR=\mcL_{G/P_s}(\bfR)$, see \cite[Example 3.5]{CFP16}.
\end{example}

Next, we discuss some Borel--Weil--Bott theorems that will be used later.

\begin{theorem}\label{thm:GLn-BWB}
Let $\mcR\subset \mcF$ be the tautological subbundle over $\Gr(s,\bfF)$. Consider the short exact sequence:
$$
0\lr \mcR\lr \mcF\lr \mcQ\lr 0.
$$
Let $\lambda$ and $\mu$ be weakly decreasing sequences of integers of lengths $s$ and $n-s$. Let $\mcV$ be the vector bundle $S_\lambda(\mcR^\vee)\otimes S_\mu(\mcQ^\vee)$ on $\Gr$.
Set $\beta=(\lambda_1,\cdots,\lambda_s,\mu_1,\cdots,\mu_{n-s})$.
Then
\begin{enumerate}
\item If $\beta$ is weakly decreasing, then 
$$
H^0(\Gr,\mcV)=S_\beta(\bfF^\vee)
$$
as $\GL(\bfF)$-representations, and all higher cohomology of $\mcV$ vanishes.
\item When $\mu=0$ and $-(n-s)\leq \lambda_s\leq -1$, then $H^i(\Gr,\mcV)=0$ for all $i$.

\item When $\mu=(1)$, or $(1,1)$, and $\lambda_s=0$, then $H^i(\Gr,\mcV)=0$ for all $i$.
\end{enumerate}
\end{theorem}
\begin{proof}
Part (1) is \cite[Theorem 2.5.1]{SSam16}. To proceed, we use \cite[Corollary 4.1.12]{Weyman_03}.
For any increasing chain $(i_1,\cdots,i_k)$, let
$
\Fl(i_1,\cdots,i_k)
$
be the partial flag variety parameterizing increasing filtrations
$$
\Fl(i_1,\cdots,i_k): \mcU_{i_1}\subset \cdots\subset \mcU_{i_k},\quad \operatorname{rank}(\mcU_{i})=i.
$$
Let $\Fl$ be the full flag variety with natural projection $\pi:\Fl\to \Gr(s,\bfF)$. For any $\beta$, consider
$$
\mcL_\beta=\bigotimes_{i=1}^s (\mcU_i/\mcU_{i-1})^{\otimes(-\lambda_i)}\otimes \bigotimes_{j=1}^{n-s}(\mcU_{s+j}/\mcU_{s+j-1})^{\otimes (-\mu_j)}.
$$
By \cite[Corollary 4.1.12]{Weyman_03}, we have
$$
R\pi_*\mcL_\beta=\mcV.
$$
Hence we are reduced to computing the cohomology of $\mcL_\beta$ over $\Fl$.

For part (2), consider 
$$
\Fl\overset{f}{\lr} \Fl(1,\cdots,s)\overset{g}{\lr} \Gr(s,\bfF).
$$
Let 
$$
\mcL_{\lambda}:=\bigotimes_{i=1}^s(\mcU_i/\mcU_{i-1})^{\otimes (-\lambda_i)}
$$
over $\Fl(1,\cdots,s)$. Since the fibers of $f$ are complete flag varieties, we have
$$
Rf_*\mcL_\beta=\mcL_{\lambda}\text{ and }Rg_*\mcL_{\lambda}=\mcV.
$$
Consider now the morphism 
$$
h:\Fl(1,\cdots,s)\lr \Fl(1,\cdots,s-1),
$$
whose fibers are isomorphic to $\mbP^{n-s}$, with the restriction
$$
(\mcU_s/\mcU_{s-1})|_{\mbP^{n-s}}=\mcO(-1).
$$
Hence we have fiberwise
$$
\mcL_\lambda|_{\mbP^{n-s}}\simeq \mcO(\lambda_s).
$$
Hence $Rh_*\mcL_{\lambda}=0$ since $-(n-s)\leq \lambda_s\leq -1$. Therefore,
$$
H^i(\Gr,\mcV)=H^i(\Fl(1,\cdots,s),\mcL_\lambda)=0\quad\text{for all }i.
$$

For part (3), let $a=1$ or $2$. Consider
$$
\Fl\overset{f}{\lr} \Fl(1,\cdots,s,s+a)\overset{g}{\lr} \Gr(s,\bfF).
$$
Let
$$
\mcL'=\bigotimes_{i=1}^s(\mcU_i/\mcU_{i-1})^{\otimes (-\lambda_i)}\otimes \det(\mcU_{s+a}/\mcU_s)^\vee
$$
on $\Fl(1,\cdots,s,s+a)$, then $Rg_*\mcL'=\mcV$. Consider
$$
h:\Fl(1,\cdots,s,s+a)\lr \Fl(1,\cdots,s-1,s+a)
$$
whose fibers are $\mbP^a$, and that
$$
\mcL'|_{\mbP^a}\simeq \mcO(\lambda_s)\otimes \mcO(-1)=\mcO(-1).
$$
Hence $Rh_*\mcL'=0$ and we conclude the assertion.
\end{proof}

\subsection{Symmetric determinantal varieties}
We recall some properties of symmetric determinantal varieties. 
We keep the notation of the previous subsection. Recall that $\Gr=\Gr(s,\bfF)$ is the Grassmannian, with universal tautological bundle $\mcR\subset \mcF$.

Let $\bfN=\Sym^2\bfF$ be the space of symmetric matrices, to distinguish it from the vector space, let $N$ be the corresponding algebraic variety over $k$.

Let $N_s\subset N$ be the subvariety consisting of symmetric matrices of rank $\leq s$. Let $\mcS:=\Sym^2\mcR$ with total space $Y_s\subset N\times \Gr\to \Gr$. By definition, it consists of symmetric matrices of rank $\leq s$ such that $X\bfF\subseteq \bfR$. Consider the following diagram:
\begin{equation*}
\begin{aligned}
\xymatrix{
Y_s\ar@{^(->}[r]\ar[d]^-{\pi_s}&N\times \Gr\ar[r]\ar[d]&\Gr\\
N_s\ar@{^(->}[r]&N
}
\end{aligned}
\end{equation*}
We recall the following result.
\begin{theorem}\label{thm:sym-determ}
Let $k[N]=\Sym(\Sym^2(\bfF^\vee))$ and let $I_{s+1}\subseteq k[N]$ be the ideal generated by the $(s+1)\times (s+1)$ minors. 
\begin{enumerate}
\item The subvariety $N_s$ is normal with defining ideal $I(N_s)=I_{s+1}\subset k[N]$.

\item The projection $\pi_s:Y_s\to N_s$ is surjective and birational. Thus there is an isomorphism of graded algebras
$$
\Gamma(Y_s,\mcO_{Y_s})=\bigoplus_{d\geq 0}\Gamma(\Gr,\Sym^d\mcS^\vee)\overset{\sim}{\lr} k[N_s].
$$
\end{enumerate}
\end{theorem}
\begin{proof}
Part (1) follows from \cite[\S 6.3]{Weyman_03} when $\mathrm{char}(k)=0$, and from \cite{dCP76} in general. Part (2) follows from the standard arguments of \cite[\S 6.3]{Weyman_03} which also apply in positive characteristic.
\end{proof}

\subsection{Isotropic Grassmannian}
Let $(\bfF,q)$ be a nondegenerate quadratic space of dimension $n\geq 2s+1$. 
Recall that we have the Grassmannian $\Gr=\Gr(s,\bfF)$ with tautological bundle $\mcR\subseteq\mcF:=\bfF\otimes_{k}\mcO_{\Gr}$. 

Let $\IGr=\IGr(s,\bfF)$ be the isotropic Grassmannian, defined as the closed subvariety of $\Gr$ consisting of isotropic subspaces. It is a Grassmannian for the orthogonal group associated with a maximal parabolic subgroup. Let $\iota:\IGr\hookrightarrow \Gr$ be the closed immersion. We use the same symbols $\mcR$ and $\mcF$ for their restrictions to $\IGr$.

The symmetric form $q$ induces a symmetric form on the Grassmannian
$$
q:\mcF\times \mcF\lr \mcO_{\Gr},
$$
where $\mcF$ is the universal vector bundle, lifting $\bfF$. By restricting to $\mcR\subset \mcF$, we get
$$
q:\mcR\times \mcR\lr \mcO_{\Gr}.
$$
Since $q$ is symmetric, it factors through the symmetric power $\Sym^2\mcR$. This defines a global section $q|_{\mcR}\in H^0(\Gr,\Sym^2\mcR^\vee)$.

By definition, the regular section $q|_{\mcR}$ vanishes on $\IGr(s,\bfF)$. Recall the following
\begin{proposition}\label{prop:OGr-lci}
The closed immersion $\iota:\IGr\subseteq \Gr$ is a local complete intersection, and $\IGr$ is the scheme-theoretic zero locus of 
$$
q|_{\mcR}\in \Gamma(\Gr,\Sym^2\mcR^\vee).
$$
The section is regular and gives the following exact Koszul resolution on $\Gr$:
$$
0\lr \bigwedge^{\binom{s+1}{2}}\bigl(\Sym^2\mcR\bigr)\lr\cdots\lr\bigwedge^2(\Sym^2\mcR\bigr)\lr\Sym^2\mcR\overset{\psi} {\lr}\mcO_{\Gr}\lr \iota_*\mcO_{\IGr}\lr 0.
$$
\end{proposition}
\begin{proof}
When $\dim \bfF$ is odd, this is \cite[Proposition 4.3.8]{Weyman_03}, when $\dim \bfF$ is even, this is \cite[Proposition 4.3.10]{Weyman_03}.
\end{proof}

The following observation plays a crucial role:
\begin{proposition}\label{prop:sym-reduction}
Suppose that $n\geq 2s+1$, and that $p=0$ or $p\geq 2s-1$.
Let $\mcS=\Sym^2\mcR$ on $\Gr$. For $d\geq 0$, $j\geq 0$, and $i> 0$, we have
$$
H^i(\Gr,\bigwedge^j\mcS\otimes \Sym^d\mcS^\vee)=0.
$$
Applying this vanishing to Proposition \ref{prop:OGr-lci}, after tensoring with $\Sym^d\mcS^\vee$ and taking hypercohomology, we obtain the following long exact sequence:
\begin{multline}\label{equ:reduction-les}
0\lr H^0(\Gr,\bigwedge^{\binom{s+1}{2}}\mcS\otimes\Sym^d\mcS^\vee)\lr H^0(\Gr,\bigwedge^{\binom{s+1}{2}-1}\mcS\otimes\Sym^d\mcS^\vee)\lr\cdots\\
\lr H^0(\Gr,\mcS\otimes\Sym^d\mcS^\vee)\lr H^0(\Gr,\Sym^d\mcS^\vee)\lr H^0(\IGr,\Sym^d\mcS^\vee)\lr 0.
\end{multline}
\end{proposition}
\begin{proof}
We use Proposition \ref{prop:exact-fun} and Example \ref{exmp:grassmannian}. Let $\bfR\subset \bfF$ be a subspace of dimension $s$. Consider the $P_s$-module 
$$
\bigwedge^j(\Sym^2 \bfR)\otimes \Sym^d(\Sym^2\bfR^\vee),
$$ 
with the action given by $P_s\twoheadrightarrow\GL(\bfR)$.
The sheaf considered in the proposition is obtained by applying the functor $\mcL_{G/P_s}(-)$. 

By Proposition \ref{prop:wedge-good}, the $\GL(\bfR)$-representation $\bigwedge^j(\Sym^2 \bfR)\otimes \Sym^d(\Sym^2\bfR^\vee)$ admits a good filtration, with associated graded part given by Schur functors. 
Hence by Proposition \ref{prop:exact-fun}, the sheaf $\bigwedge^j\mcS\otimes \Sym^d\mcS^\vee$ admits a filtration, with associated graded part given by $S_{\lambda}\mcR^\vee$ with $\lambda_s\geq -(s+1)\geq -(n-s)$. By Theorem \ref{thm:GLn-BWB}, the higher cohomology of each associated graded part vanishes, hence so does that of the vector bundle itself.
\end{proof}

\section{Symmetric nilpotent orbit closure}
\subsection{Kempf's collapsing}
Let $k$ be an algebraically closed field.
Let $(\bfF,q)$ be a nondegenerate quadratic space of dimension $n\geq 2s+1$ over the field $k$. Consider 
$$
\bfN:=\{X\in\End(\bfF):X^t=X\}\simeq \Sym^2 \bfF.
$$
This defines a vector space over $k$, let $N$ be the corresponding affine scheme over $k$.
Let $u=(2^s,1^{n-2s})$ be a partition of $n$ and let $N_{u}\subset N$ be the closed subvariety whose geometric points are
$$
N_{u}(k)=\{X\in \bfN(k)\mid X^2=0, \mathrm{rank}(X)\leq s\}.
$$
This is the nilpotent orbit closure for the symmetric pair considered in \cite[\S 5.4]{PR09}.

We consider Kempf's collapsing \cite{Kempf} in this case.
Let $\IGr=\IGr(s,\bfF)$ be the isotropic Grassmannian with universal subbundle $\mcR\subset \mcF$. 
Let $Y_{u}\subset \bfN\times \IGr(s,\bfF)$ be the closed subvariety whose geometric points are
$$
Y_{u}(k)=\{
(X,\bfR)\mid X\bfF\subseteq \bfR, X\bfR=(0)
\}.
$$
Since $X$ is symmetric, $q(X\bfR,\bfF)=q(\bfR,X\bfF)\subseteq q(\bfR,\bfR)=(0)$. Hence the condition $X\bfR=(0)$ is redundant.

\begin{proposition}\label{prop:birational}
\begin{enumerate}
\item The composition 
$$
Y_{u}\to N\times \IGr(s,\bfF)\to \IGr(s,\bfF)
$$
is the total space of $\mcS:=\Sym^2\mcR$, the symmetric power of the universal subbundle.
\item The composition
$$
Y_{u}\to N\times \IGr(s,\bfF)\to N
$$
factors through $N_u$.
\item The morphism $\pi:Y_u\to N_u$ is dominant and birational.
\end{enumerate}
\end{proposition}
\begin{proof}
Part (1) follows by pulling back the $\Gr(s,\bfF)$ case of \cite[\S 6.3]{Weyman_03}.
Part (2) follows directly from the definition. 
For part (3), notice that for any $X\in N_{u}$, we have an isomorphism
$$
\pi^{-1}(X)=\IGr(s-\operatorname{rank}(X),(X\bfF)^\perp/X\bfF).
$$
In particular, $\pi$ is an isomorphism over the open orbit.
\end{proof}

We will show in Proposition \ref{prop:normal-orbit-closure} that $N_u$ is normal (when $p=0$, it follows from \cite[Theorem 1]{Trevisiol24}), by relating $N_u$ to Schubert varieties in the affine Grassmannian. Hence $\Gamma(Y_u,\mcO_{Y_u})=\Gamma(N_u,\mcO_{N_u})$ by Proposition \ref{prop:birational}.

Consider now the composition
$$
Y_u\lr N_u\lr N.
$$
Our goal then reduces to finding the kernel of the morphism
$$
\Gamma(N,\mcO_N)\twoheadrightarrow\Gamma(Y_u,\mcO_{Y_{u}})=\bigoplus_{d\geq 0}H^0(\IGr,\Sym^d(\Sym^2\mcR^\vee)).
$$

\begin{remark}
It is possible to find the defining equations for the nilpotent orbit closure directly. Over the complex numbers, this would also follow from the basic theorem in \cite[\S 5.1]{Weyman_03}, since it is possible to calculate all cohomology groups explicitly.
\end{remark}

\subsection{The main theorem}
Recall that $\mcS=\Sym^2\mcR$.
Consider the graded $k[N_s]$-module
$$
\bfM=\bigoplus_{d\geq 0}\bfM_d=\bigoplus_{d\geq 0} H^0(\Gr,\mcS\otimes \Sym^d\mcS^\vee).
$$
By Proposition \ref{prop:sym-reduction}, we have a right exact sequence
$$
\bfM\overset{\alpha}{\lr} k[N_s]\lr k[N_u]\lr 0.
$$
We now introduce two families of elements of $\bfM$.

\begin{enumerate}
\item Let $\tau\in\Gamma(\Gr,\mcS\otimes\mcS^\vee)=\bfM_1$ be the identity element. 
\item Consider the homomorphism of vector bundles
$$
\Phi:\mcS^\vee \lr \mcS\otimes \Sym^2(\mcS^\vee),\quad U\mapsto \bigl[X\mapsto XUX \bigr],
$$
here we view an element of the right hand side as an element of the sheaf hom $\underline{\Hom}(\mcS^\vee,\Sym^2(\mcS^\vee))$. 
By taking global sections, we obtain
$$
\Phi:\bfN^\vee\lr \bfM_2.
$$
\end{enumerate}

In the next section, we will prove the following theorem.
\begin{theorem}\label{thm:surjective-multi}
The $k[N_s]$-module $\bfM$ is generated by $\tau\in \bfM_1$ and the sections $\Phi(\lambda)\in \bfM_2$ for $\lambda\in \bfN^\vee$. Equivalently, we have
$$
\bfM_0=0,\quad
\bfM_1=k\cdot\tau,\quad
\bfM_2=\bfN^\vee\cdot \tau+\Phi(\bfN^\vee),
$$
and for every $d\geq 3$, the multiplication map
$$
\Bigl(\Gamma(\Gr,\Sym^{d-1}\mcS^\vee)\otimes \bfM_1\Bigr)
\oplus 
\Bigl(\Gamma(\Gr,\Sym^{d-2}\mcS^\vee)\otimes \bfN^\vee\Bigr)\lr \bfM_d
$$
is surjective.
\end{theorem}

\begin{lemma}\label{lem:alpha-calculate}
\begin{enumerate}
\item We have $\alpha(\tau)=\tr(X)\in k[N_s]_1$;
\item For any $\lambda\in \bfN^\vee$, we have
$$
\alpha(\Phi(\lambda))=\lambda(X^2)\in k[N_s]_2
$$
as functions on $N_s$.
\end{enumerate}
\end{lemma}
\begin{proof}
Recall from Proposition \ref{prop:OGr-lci} that $\alpha$ is defined by evaluating the quadratic form $q|_{\mcR}:\mcS\to \mcO$. Fix an orthonormal basis $e_1,\cdots,e_n$ so that $q(e_i,e_j)=\delta_{ij}$.
For part (1), we have $\alpha(\sum_{ij}x_{ij}\otimes x_{ij}^\vee)=\sum_{i}x_{ii}^\vee=\tr(X)$.
For part (2), we have
$$
\Phi(E_{ij}+E_{ji})=X\mapsto (x_{ai}x_{jb}+x_{aj}x_{ib})_{ab},
$$
hence $\alpha(\Phi(E_{ij}+E_{ji}))=2\sum_k x_{ik}x_{kj}=2(X^2)_{ij}$ (note that $X=X^t$). The assertion follows by spanning.
\end{proof}

Now we are ready to prove the main theorem.
\begin{theorem}\label{thm:main-equation}
Suppose $p=0$, or $p\geq 2s-1$ ($p>2$). Let $(\bfF,q)$ be a nondegenerate quadratic space of dimension $n\geq 2s+1$. The subvariety $N_u\subset N$ has defining ideal
$$
I(N_u)=\bigl(
\tr(X),(X^2)_{ij},I_{s+1}(X)
\bigr)\subset k[N].
$$
\end{theorem}
\begin{proof}
By Theorem \ref{thm:surjective-multi} and Lemma \ref{lem:alpha-calculate}, the kernel
$$
\ker(k[N_s]\lr k[N_u])\cong \alpha(\bfM)
$$
is generated by $\tr(X)$ and $(X^2)_{ij}$. Since $N_s\subset N$ has defining ideal $I(N_s)=(I_{s+1}(X))\subset k[N]$, the assertion follows.
\end{proof}

\section{The multiplication maps}\label{sec:multi-maps}
In this section, we complete the proof of the main result by proving Theorem \ref{thm:surjective-multi}. We keep the notation of the previous section, in particular, $\Gr=\Gr(s,\bfF)$ with tautological subbundle $\mcR$, and $\mcS=\Sym^2\mcR$. Recall that
$$
\bfM_d=\Gamma(\Gr,\mcS\otimes \Sym^d(\mcS^\vee)).
$$

\subsection{Modular representation calculation}\label{sec:multi-maps-rep}
We first study the analogous problem for representations of $\GL_s(k)$. Let $\bfR$ be a vector space over $k$ of dimension $s$. Fix a basis $x_1,\cdots,x_s$ of $\bfR$, with dual basis $x_1^\vee,\cdots,x_s^\vee$ of $\bfR^\vee$. Let
$$
\bfP:=\bigoplus_{d\geq 0}\bfP_d,\quad \bfP_d:=\Sym^d(\Sym^2(\bfR^\vee)).
$$
Consider
$$
\tau=\sum_{a,b}x_ax_b\otimes x_a^\vee x_b^\vee\in \Sym^2\bfR\otimes \Sym^2(\bfR^\vee).
$$
And 
$$
\Phi:\Sym^2 \bfR^\vee\otimes \bfP_{d-2}\lr \Sym^2 \bfR\otimes \bfP_d,\quad u_1u_2\mapsto \sum_{a,b}x_ax_b\otimes (x_a^\vee u_1)(x_b^\vee u_2),\quad u_1,u_2\in \bfR^\vee.
$$
It is a $\bfP$-linear map of degree $2$.
Consider 
$$
\Psi_d: \bfP_{d-1}\oplus (\Sym^2 \bfR^\vee\otimes \bfP_{d-2})\lr \Sym^2 \bfR\otimes \bfP_d,\quad (f,b)\mapsto f\tau+\Phi(b).
$$

\begin{lemma}\label{lem:good-fil-surj}
Assume $p\neq 2$. Then 
\begin{enumerate}
\item The map $\Phi$ is injective.
\item Let $\det(X)\in \bfP_s$ be the determinant and let $\adj\in \Sym^2(\bfR^\vee)\otimes \bfP_{s-1}$ be the adjoint matrix. Then we have $\GL(\bfR)$-equivariant isomorphisms
$$
(\det \bfR^\vee)^{\otimes 2}\simeq k\cdot \det\subset \bfP_s,\quad
(\det \bfR^\vee)^{\otimes 2}\simeq k\cdot \adj\subset \Sym^2 \bfR^\vee\otimes \bfP_{s-1}. 
$$
\item We have a left exact sequence
$$
0\lr \bfP_{d-s-1}\otimes (\det \bfR^\vee)^{\otimes 2}\overset{(\det,-\adj)^t}\lr \bfP_{d-1}\oplus (\Sym^2 \bfR^\vee\otimes \bfP_{d-2})\overset{\Psi_d}{\lr} \Sym^2\bfR\otimes \bfP_d.
$$
Therefore, the image $\im(\Psi_d)$ has a good filtration.

\item The formal character
$$
\ch(\im\Psi_d)=\ch \bfP_{d-1}+\ch(\Sym^2 \bfR^\vee)\cdot \ch \bfP_{d-2}-\ch((\det \bfR^\vee)^{\otimes 2})\cdot \ch \bfP_{d-1-s},
$$
is independent of $p$.
\end{enumerate}
\end{lemma}
\begin{proof}
Part (1) is straightforward from the coordinate expression. For part (2), for any $g\in \GL(\bfR)$, we have
$$
(g\cdot \det)(X)=\det(g^{-1}Xg^{-t})=\det(g)^{-2}\det(X).
$$
And (when $X$ is invertible) we have
$$
(g\cdot \adj)(X)=g^{-t}\adj(g^{-1}Xg^{-t})g^{-1}=g^{-t}\det(g^{-1}Xg^{-t})(g^{-1}Xg^{-t})^{-1}g^{-1}=\det(g)^{-2}\adj(X),
$$
hence both $k\cdot \det$ and $k\cdot \adj$ are isomorphic to $(\det \bfR^\vee)^{\otimes 2}$.

For part (3), it is straightforward to see that the kernel contains the image of $(\det,-\adj)$.
Suppose 
$$
(f,b)\in \ker\Psi\subset \bfP_{d-1}\oplus (\Sym^2 \bfR^\vee\otimes \bfP_{d-2}),$$
then we have
$$
fX+XbX=0\text{ for all } X\in \Sym^2(\bfR).
$$
This is an identity of symmetric matrices whose entries are polynomials. Multiplying on the right by $\adj(X)$, we get
$
(f+Xb)\det(X)=0,
$
hence $fI+Xb=0$, and $\adj(X)f=-\adj(X)Xb=-\det(X)b$.
Since $\det(X)$ is coprime to the entries of $\adj(X)$ (for symmetric matrices $X$), we conclude that $\det(X)\mid f$ and we obtain a polynomial $g:=\det(X)^{-1}f$. 
Then $Xb=-\det(X)g$, multiplying on the left by $\adj(X)$, we obtain
$
\det(X)b=-\det(X)\adj(X)g,
$
hence $\det(X)(b+\adj(X)g)=0$, and the assertion follows. For part (4), note that $\bfP_d$ has a universal filtration by Theorem \ref{thm:schur-decomp}, hence the formal characters are independent of $p$.
\end{proof}

\begin{proposition}\label{prop:char0-calculation}
Suppose $p=0$. Then for any $d\geq 0$, the subspace $\im(\Psi_d)\subset \Sym^2 \bfR\otimes \bfP_d$ is the sum of the isotypic components $S_\lambda(\bfR^\vee)$ with $\lambda_s\geq 0$.
\end{proposition}
\begin{proof}
We may assume $k=\mbC$. 
Let $\wt{\bfM}_d\subset \Sym^2 \bfR \otimes \bfP_d$ be the sum of the isotypic components $S_{\lambda}(\bfR^\vee)$ with $\lambda_s\geq 0$. 
It is straightforward to verify that $\im(\Psi_d)\subseteq \wt{\bfM}_d$, we check the opposite direction.

Let $S_{2d}$ be the symmetric group and let $H_d\subset S_{2d}$ be the centralizer of $(12)(34)\cdots(2d-1,2d)$. Consider the idempotents
$$
e_1:=\frac{1}{2}(1+(12))\text{ on }(\bfR)^{\otimes 2},\quad
e_2^{(d)}:=\frac{1}{|H_d|}\sum_{h\in H_d} h\text{ on }(\bfR^\vee)^{\otimes 2d}.
$$
We have $\GL(\bfR)$-equivariant identifications:
$$
\Sym^2 \bfR=(\bfR^{\otimes 2})^{S_2}=e_1(\bfR^{\otimes 2}),\quad \bfP_d=\Sym^d(\Sym^2 \bfR^\vee)=((\bfR^\vee)^{\otimes 2d})^{H_d}=e_2^{(d)}((\bfR^\vee)^{\otimes 2d}),
$$
with the monomial $(u_1u_2)\cdots(u_{2d-1}u_{2d})$ corresponding to $e_2^{(d)}(u_1\otimes\cdots\otimes u_{2d})$. 
Consider the idempotent $e^{(d)}:=e_1\otimes e_2^{(d)}$ on $W_d:=\bfR^{\otimes 2}\otimes (\bfR^\vee)^{\otimes 2d}$, it is an equivariant projection of $W_d$ onto $\Sym^2 \bfR\otimes \bfP_d$. Since $H_k\times H_l\subseteq H_{k+l}$, it follows that
\begin{equation}\label{equ:induction-formula}
e^{(k+l)}(y\otimes z)=e^{(k)}(y)\cdot e_2^{(l)}(z)\quad \text{for }y\in W_k,z\in (\bfR^\vee)^{\otimes 2l}.
\end{equation}

By the tensor realization, the irreducible representations $S_{\lambda}(\bfR^\vee)$ with $|\lambda|=2d-2$ and $\lambda_s\geq 0$ are exactly the irreducible constituents of $(\bfR^\vee)^{\otimes (2d-2)}$. By complete reducibility, $\wt{\bfM}_d$ is therefore the sum of the images of all $\GL(\bfR)$-equivariant maps $f:(\bfR^\vee)^{\otimes (2d-2)}\to \Sym^2 \bfR\otimes \bfP_d$. Consider the adjunction isomorphism
$$
c:\End_{\GL(\bfR)}((\bfR^\vee)^{\otimes 2d})\overset{\sim}{\lr} \Hom_{\GL(\bfR)}((\bfR^\vee)^{\otimes (2d-2)}, W_d),
$$
which is given by 
$$
c(F)(u^\vee)=\sum_{a,b} x_a\otimes x_b\otimes F(u^\vee\otimes x_a^\vee\otimes x_b^\vee).
$$

By Schur--Weyl duality, the space $\End_{\GL(\bfR)}((\bfR^\vee)^{\otimes 2d})$ is spanned by the permutations $\sigma\in S_{2d}$. Hence any $f\in \Hom_{\GL(\bfR)}((\bfR^\vee)^{\otimes (2d-2)},W_d)$ is given by
$$
f=\sum_{\sigma\in S_{2d}}a_\sigma c(\sigma),\quad a_\sigma\in \mbC.
$$
Suppose furthermore that $f\in \Hom_{\GL(\bfR)}((\bfR^\vee)^{\otimes (2d-2)},\wt{\bfM}_d)$, then
$$
f=e^{(d)}\circ f=\sum_{\sigma} a_\sigma e^{(d)}\circ c(\sigma).
$$
Since each $e^{(d)}\circ c(\sigma)$ is itself an equivariant map in $\Hom_{\GL(\bfR)}((\bfR^\vee)^{\otimes (2d-2)},\Sym^2 \bfR\otimes \bfP_d)$, we conclude that
$$
\wt{\bfM}_d=\sum_{\sigma\in S_{2d}}\operatorname{im}(e^{(d)}\circ c(\sigma)).
$$

We now study the image of each map explicitly. For distinct $p,q\in\{1,\cdots,2d\}$ and vectors $u_1^\vee,\cdots,u_{2d-2}^\vee\in \bfR^\vee$, let
$$
w_{p,q}(a,b)=u_1^\vee\otimes\cdots\otimes u_{p-1}^\vee\otimes x_a^\vee\otimes \cdots \otimes x_b^\vee\otimes u_{2d-2}^\vee\in (\bfR^\vee)^{\otimes 2d}
$$ 
be the pure tensor with $x_a^\vee$ in position $p$ and $x_b^\vee$ in position $q$, and $u_1^\vee,\cdots,u_{2d-2}^\vee$ in the remaining positions in increasing order. Let $I_{p,q}$ be the span of the vectors 
$$
e^{(d)}\Bigl(
\sum_{a,b}x_a\otimes x_b\otimes w_{p,q}(a,b)
\Bigr),\quad u_1^\vee,\cdots,u_{2d-2}^\vee\in \bfR^\vee.
$$
We have
\begin{align*}
c(\sigma)(u_1^\vee\otimes \cdots \otimes u_{2d-2}^\vee)={}&\sum_{a,b}x_a\otimes x_b\otimes \sigma(u_1^\vee\otimes \cdots\otimes u_{2d-2}^\vee\otimes x_a^\vee\otimes x_b^\vee)\\
={}&\sum_{a,b}x_a\otimes x_b\otimes w_{\sigma(2d-1)\sigma(2d)}(a,b).
\end{align*}
Since the $u_i^\vee$ range over all of $\bfR^\vee$, we see that 
$$
\operatorname{im}(e^{(d)}\circ c(\sigma))=I_{\sigma(2d-1)\sigma(2d)}.
$$

Moreover, $I_{p,q}=I_{q,p}$ since $e_1(x_b\otimes x_a)=e_1(x_a\otimes x_b)$. For any $h\in H_d$, we have $I_{p,q}=I_{h(p),h(q)}$ since $e_2^{(d)}\circ h=e_2^{(d)}$. Therefore, we have
$$
\wt{\bfM}_d=\sum_{p,q}I_{p,q}=I_{1,2}+I_{1,3}.
$$
It remains to compute $I_{1,2}$ and $I_{1,3}$.

First, we have $I_{1,2}=\tau\cdot \bfP_{d-1}$. By \eqref{equ:induction-formula} with $(k,l)=(1,d-1)$, we have
$$
e^{(d)}\Bigl(
x_a\otimes x_b\otimes x_a^\vee\otimes x_b^\vee\otimes u_1^\vee\otimes \cdots \otimes u_{2d-2}^\vee
\Bigr)=\tau\cdot e_2^{(d-1)}(u_1^\vee\otimes \cdots\otimes u_{2d-2}^\vee),
$$
and we have $e_2^{(d-1)}((\bfR^\vee)^{\otimes (2d-2)})=\bfP_{d-1}$.

Next, we have $I_{1,3}=\Phi(\Sym^2 \bfR^\vee\otimes \bfP_{d-2})$. Let $z^\vee:=u_3^\vee\otimes \cdots \otimes u_{2d-2}^\vee$, by \eqref{equ:induction-formula} with $(k,l)=(2,d-2)$, we have
$$
e^{(d)}\Bigl(
\sum_{a,b}x_a\otimes x_b\otimes x_a^\vee\otimes u_1^\vee\otimes x_b^\vee\otimes u_2^\vee\otimes z^\vee
\Bigr)=\Phi(u_1^\vee u_2^\vee)\cdot e_2^{(d-2)}(z^\vee).
$$
As $u_1^\vee,u_2^\vee$ range over $\bfR^\vee$, the products $u_1^\vee u_2^\vee$ span $\Sym^2 \bfR^\vee$, and $e_2^{(d-2)}((\bfR^\vee)^{\otimes (2d-4)})=\bfP_{d-2}$, which proves the claim and hence the proposition.
\end{proof}

\begin{corollary}\label{cor:charp-quotient}
Suppose $p\neq 2$. For any $d\geq 0$, the quotient 
$$
\Sym^2 \bfR\otimes \bfP_d\twoheadrightarrow \coker(\Psi_d)
$$ 
admits a good filtration whose associated graded parts are $S_{\gamma}(\bfR^\vee)$ with $\gamma_s<0$.
\end{corollary}
\begin{proof}
By Lemma \ref{lem:good-fil-surj}, we know that $\im(\Psi_d)$ has good filtration. Since $\Sym^2 \bfR\otimes \bfP_d$ has good filtration, we see that the cokernel $\coker(\Psi_d)$ has good filtration. The contribution of $S_\lambda(\bfR^\vee)$ can be calculated from the formal character
$$
\ch(\coker(\Psi_d))=\ch(\Sym^2\bfR)\cdot \ch(\bfP_d)-\ch(\im(\Psi_d)),
$$
which is independent of $p$. Hence the assertion follows from Proposition \ref{prop:char0-calculation}.
\end{proof}

\subsection{Surjectivity of the global sections}
Now we are ready to prove the main result. Consider the short exact sequence
$$
0\lr \bfP_{d-s-1}\otimes (\det \bfR^\vee)^{\otimes 2}\overset{(\det,-\adj)^t}{\lr} \bfP_{d-1}\oplus (\Sym^2 \bfR^\vee\otimes \bfP_{d-2})\overset{\Psi_d}{\lr}\im(\Psi_d)\lr 0.
$$
Applying the functor $\mcL_{G/P_s}(-)$ in Proposition \ref{prop:exact-fun}, we obtain the following exact sequence over $\Gr$ (recall that $\mcS=\Sym^2\mcR$):
$$
0\lr \Sym^{d-s-1}(\mcS^\vee)\otimes (\det \mcR^\vee)^{\otimes 2}\overset{(\det,-\adj)^t}{\lr} \Sym^{d-1}(\mcS^\vee)\oplus (\mcS^\vee\otimes \Sym^{d-2}(\mcS^\vee))\overset{\Psi_d}{\lr}\mcL_{G/P_s}(\im(\Psi_d))\lr 0,
$$
where by the exactness, $\mcL_{G/P_s}(\im(\Psi_d))$ is identical to
$$
\im\Psi_{d}:=\im\Bigl(
\Psi_d: \Sym^{d-1}(\mcS^\vee)\oplus (\mcS^\vee\otimes \Sym^{d-2}(\mcS^\vee))\lr \mcS \otimes \Sym^{d}(\mcS^\vee)
\Bigr).
$$

\begin{proposition}
Suppose $p\neq 2$ and $n\geq 2s+1$.
\begin{enumerate}
\item The inclusion $\im(\Psi_d)\subseteq \mcS\otimes \Sym^d(\mcS^\vee)$ induces an isomorphism
$$
\Gamma(\Gr,\im(\Psi_d))\overset{\sim}{\lr} \Gamma(\Gr,\mcS \otimes \Sym^d(\mcS^\vee))=\bfM_d.
$$
\item We have a surjection 
$$
\Psi_d: \Gamma(\Gr,\Sym^{d-1}(\mcS^\vee))\oplus \Gamma(\Gr,\mcS^\vee\otimes \Sym^{d-2}(\mcS^\vee))\lr \Gamma(\Gr,\im(\Psi_d)).
$$
\end{enumerate}
\end{proposition}
\begin{proof}
For part (1), consider the short exact sequence of $\GL(\bfR)$-modules
$$
0\lr \im(\Psi_d)\lr \Sym^2\bfR\otimes \Sym^d(\Sym^2(\bfR^\vee))\lr \coker(\Psi_d)\lr 0.
$$
By Lemma \ref{lem:good-fil-surj}, the first two terms have good filtrations, hence so does $\coker(\Psi_d)$. Therefore, the vector bundle $\mcL_{G/P_s}(\coker(\Psi_d))$ admits a filtration with associated graded part of the form $S_{\lambda}(\mcR^\vee)$.
By Corollary \ref{cor:charp-quotient}, all associated graded parts of $\coker(\Psi_d)$ are of the form $S_{\gamma}(\bfR^\vee)$ with $\gamma_s<0$. By Proposition \ref{prop:wedge-good}, we also have $\gamma_s\geq -(s+1)\geq -(n-s)$. Then all cohomology groups of $\mcL_{G/P_s}(\coker(\Psi_d))$ vanish by Theorem \ref{thm:GLn-BWB}, this proves (1).

For part (2), applying $\mcL_{G/P_s}(-)$ to the short exact sequence
$$
0\lr \bfP_{d-s-1}\otimes (\det \bfR^\vee)^{\otimes 2}\overset{(\det,-\adj)^t}\lr \bfP_{d-1}\oplus (\Sym^2 \bfR^\vee\otimes \bfP_{d-2})\overset{\Psi_d}{\lr} \im(\Psi_d)\lr 0,
$$
we reduce to prove that $H^1(\Gr,\mcL_{G/P_s}(\bfP_{d-s-1}\otimes (\det \bfR^\vee)^{\otimes 2}))=0$. The term $\bfP_{d-s-1}\otimes (\det \bfR^\vee)^{\otimes 2}$ has a good filtration, with each factor $S_{\lambda}(\bfR^\vee)$ satisfying $\lambda_s\geq 2$, hence the vanishing holds by Theorem \ref{thm:GLn-BWB}.
\end{proof}

We are ready to prove the surjectivity.
\begin{proof}[Proof of Theorem \ref{thm:surjective-multi}]
Suppose $p\neq 2$. We claim that the natural map
$$
\bfN^\vee\otimes \Gamma(\Gr,\Sym^{d-2}(\mcS^\vee))\twoheadrightarrow
\Gamma(\Gr,\mcS^\vee\otimes \Sym^{d-2}(\mcS^\vee))
$$
is surjective, then the composition 
$$
\Gamma(\Gr,\Sym^{d-1}(\mcS^\vee))\oplus 
\Bigl(\bfN^\vee\otimes \Gamma(\Gr,\Sym^{d-2}(\mcS^\vee))\Bigr)
\lr \Gamma(\Gr,\im(\Psi_d))=\bfM_d
$$
is a surjection, and the theorem follows. 

Applying Schur functor to the short exact sequence 
$$
0\lr \mcQ^\vee \lr\mcF^\vee \lr \mcR^\vee\lr 0.
$$
By Theorem \ref{thm:Schur-complex}, we obtain a long exact sequence
$$
0\lr \bigwedge^2 \mcQ^\vee \lr \mcQ^\vee\otimes \mcF^\vee \lr \Sym^2\mcF^\vee\lr \mcS^\vee \lr 0.
$$
This breaks into two short exact sequences:
$$
0\lr \bigwedge^2 \mcQ^\vee \lr \mcQ^\vee\otimes \mcF^\vee \lr \mcK\lr 0
\quad\text{and}\quad
0\lr \mcK \lr \Sym^2\mcF^\vee\lr \mcS^\vee \lr 0.
$$
The claim follows from $H^1(\Gr,\mcK\otimes \Sym^d\mcS^\vee)=0$, which would follow from 
$$
H^1(\Gr,\mcQ^\vee\otimes \mcF^\vee\otimes \Sym^d\mcS^\vee)=0
\quad\text{and}\quad
H^2(\Gr,\bigwedge^2\mcQ^\vee\otimes \Sym^d\mcS^\vee)=0.$$ 
Since there is a filtration
$$
\Sym^d(\Sym^2\mcR^\vee)\approx \bigoplus_{\lambda} S_{\lambda}\mcR^\vee,
$$
we see that $\Sym^d(\Sym^2\mcR^\vee)\otimes \mcQ^\vee$ (resp. $\Sym^d(\Sym^2\mcR^\vee)\otimes \bigwedge^2\mcQ^\vee$) admits a filtration with graded parts of the form $S_{\lambda}\mcR^\vee\otimes \mcQ^\vee$ (resp. $S_{\lambda}\mcR^\vee\otimes \bigwedge^2\mcQ^\vee$). By Theorem \ref{thm:GLn-BWB}(1)(3), all higher cohomology groups vanish, this finishes the proof.
\end{proof}

\section{Applications to unitary Shimura varieties}
In this section, we apply our result to ramified unitary local models and Shimura varieties.

\subsection{Wedge local model and Pappas's conjecture}\label{sec:Pappasconj}
Let $F/F_0$ be a ramified quadratic extension of $p$-adic fields $(p>2)$, with uniformizers $\pi\in F$ and $\pi_0\in F_0$ such that $\pi^2=\pi_0$. Let $k_F$ be the residue field of $O_F$.

Let $(V,\phi)$ be an $F$-hermitian space of dimension $n$. Consider the induced alternating form $\langle-,-\rangle:V\times V\to F_0$ given by 
$$
\langle x,y\rangle:=\frac{1}{2}\tr_{F/F_0}(\pi^{-1}\phi(x,y)).
$$
Let $\Lambda\subset V$ be an $O_F$-lattice such that $\Lambda^\vee=\Lambda$. Here we can take either the hermitian dual or the alternating dual. We thus obtain a perfect alternating form
$$
\langle-,-\rangle: \Lambda\times\Lambda\lr O_{F_0}.
$$
\begin{definition}\label{def:wedgedLM}
Fix $(r,s)$ such that $r+s=n$ and $r\neq s$. The \emph{wedge local model} $\mbM^{\wedge}_{(r,s)}$ is a projective scheme over $\Spec O_F$, representing the moduli functor that assigns to each $O_F$-algebra $R$ the set of all filtrations $(\mcF\subset\Lambda\otimes_{O_{F_0}}R)$ such that:
\begin{enumerate}[label=LM \arabic*.]
\item The filtration $\mcF\subset \Lambda\otimes_{O_{F_0}}R$ is an $O_F\otimes_{O_{F_0}}R$-submodule and an $R$-direct summand of rank $n$;
\item The perfect $R$-bilinear pairing
$$
\langle-,-\rangle:\Lambda\otimes_{O_{F_0}}R\times \Lambda\otimes_{O_{F_0}}R\lr R
$$
makes $\mcF$ a totally isotropic subspace.
\item(Kottwitz condition) The element $\pi\otimes 1\in O_F\otimes_{O_{F_0}}R$ acts on $\mcF$ as an $R$-linear endomorphism with characteristic polynomial
$$
\det(T\cdot \id-\pi\otimes 1\mid \mcF)=(T-\pi)^s(T+\pi)^r\in R[T];
$$
\item(Wedge condition) We have
$$
\bigwedge^{s+1}_R\bigl(
\pi\otimes 1+1\otimes \pi\mid \mcF
\bigr)=0,\quad 
\bigwedge^{r+1}_R\bigl(
\pi\otimes 1-1\otimes \pi\mid \mcF
\bigr)=0.
$$
\end{enumerate}

We define the \emph{local model} $\mbM_{(r,s)}^{\loc}$ to be the scheme-theoretic closure of the generic fiber $\mbM^{\wedge}_{(r,s)}\otimes_{O_{F}}F\hookrightarrow \mbM_{(r,s)}^{\wedge}$.
\end{definition}

By conjugation, we have
$$
\mbM_{(r,s)}^{\wedge}\cong \mbM_{(s,r)}^{\wedge}\quad\text{and}\quad \mbM_{(r,s)}^{\loc}\cong \mbM_{(s,r)}^{\loc}.
$$
Hence without loss of generality, we may assume $r>s$, and we keep this assumption in the remaining part of the paper.

\begin{theorem}\label{thm:coherent}
\begin{enumerate}
\item The local model $\mbM_{(r,s)}^{\loc}$ and its special fiber $\mbM_{(r,s)}^{\loc}\otimes k_F$ are normal.
\item The wedge local model $\mbM^{\wedge}_{(r,s)}$ is topologically flat; in other words, the underlying topological spaces of $\mbM^{\loc}_{(r,s)}$ and $\mbM^{\wedge}_{(r,s)}$ coincide.
\item We have $\mbM_{(r,s)}^{\loc}=\mbM_{(r,s)}^{\wedge}$ if and only if the special fiber $\mbM^\wedge_{(r,s)}\otimes k_F$ is reduced.
\end{enumerate}
\end{theorem}
\begin{proof}
Part (1) is a direct consequence of the coherence conjecture, see \cite[Theorem 11.3]{PR2008}. The conjecture was proposed in loc.cit. and proved in \cite{Zhu_coherent}. See also \cite{GL_tubular} for a simplified modern proof.
For part (2), when $n$ is odd, this follows from \cite{Smithling_odd}. When $n$ is even, this follows from \cite[Remark 7.4.9]{Smithling_even}.
For part (3), this follows from \cite[Proposition 14.17]{GW_AGI}, combined with parts (1) and (2).
\end{proof}

\begin{remark}
Theorem \ref{thm:coherent}(2) also follows from a direct construction of the Schubert cells. For any $0\leq i\leq \mathrm{min}(r,s)$, we define $C_i\subset \mbM_{(r,s)}^\wedge\otimes k_F$ as the locally closed subvariety whose geometric points are
$$
C_i(k)=\Bigl\{\mcF\subset\Lambda\otimes_{O_{F_0}}k\mid \mathrm{rank}(\pi\otimes 1)\mcF=i
\Bigr\}.
$$
This stratification is stable under the action of the loop group, and agrees with the Schubert cell decomposition of the special fiber of the local model $\mbM^{\loc}_{(r,s)}\otimes k_F$.
\end{remark}

By Theorem \ref{thm:coherent}, the (geometric) special fibers of the wedge local model and the local model have the same underlying topological space, in particular, both contain the unique closed Schubert cell, known as the ``worst point''. 
Therefore, to study the geometry of $\mbM^\wedge_{(r,s)}$, it suffices to restrict to an open affine neighborhood of the worst point.

The worst point is given by 
$$
\Bigl((\pi\otimes 1)\Lambda\otimes_{O_{F_0}}k_F\subseteq \Lambda\otimes_{O_{F_0}}k_F\Bigr)\in \mbM^\wedge_{(r,s)}(k_F).
$$
Consider the natural embedding of the wedge local model into the Grassmannian
$$
\mbM^{\wedge}_{(r,s)}\hookrightarrow \Gr(n,\Lambda\otimes_{O_{F_0}}O_F).
$$
After fixing an orthonormal $O_F$-basis for $\Lambda$ so that $\phi$ is represented by identity matrix (after passing to some unramified extension), we restrict to an affine chart $U\subset \Gr(n,\Lambda\otimes_{O_{F_0}}O_F)$. A point in the chart $U$ can be represented by a $2n\times n$ matrix 
$
\begin{pmatrix}
X\\ I_n
\end{pmatrix}.
$
Assume $r>s$. In this case, the special fiber of the affine chart of the wedge local model admits the following identification:
\begin{equation}\label{equ:embedding-of-chart}
(\mbM^\wedge_{(r,s)}\cap U)\otimes k_F\overset{\sim}{\lr}  \Spec \frac{k_F[X]}{X-X^t,\tr(X),X^2,\bigwedge^{s+1}(X)}.
\end{equation}
To be more precise, let $\Pi:=\pi\otimes 1$. Over this chart, $\Pi$ acts on $\mcF$ via the matrix $X$. The isotropy condition LM 2 implies $X=X^t$. Since $\Pi^2=\pi_0\otimes 1=1\otimes \pi_0$ vanishes over the special fiber, we have $X^2=0$. The Kottwitz condition implies $\tr(X)=0$ (here we use Newton's identity $(p>s)$ to eliminate the other equations arising from the Kottwitz condition.), and the wedge condition implies $\bigwedge^{s+1}(X)=0$. We refer the reader to \cite[\S 3]{Luo_LM} for a similar calculation.

\begin{proposition}\label{prop:normal-orbit-closure}
For $p>2$, the nilpotent orbit closure $N_{u}$ is normal.
\end{proposition}
\begin{proof}
Now the reduced locus of the right hand side of \eqref{equ:embedding-of-chart} is the nilpotent orbit closure $N_{u}$, while the reduced locus of the left hand side is $(\mbM^{\loc}_{(r,s)}\cap U)\otimes k_F$ by Theorem \ref{thm:coherent}(2), which is normal by Theorem \ref{thm:coherent}(1).
\end{proof}

\begin{theorem}\label{thm:Pappas-rs}
Let $r>s$ and suppose $p\geq 2s-1$. We have $\mbM_{(r,s)}^{\loc}=\mbM_{(r,s)}^{\wedge}$.
\end{theorem}
\begin{proof}
Assume $r>s$ as before.
When $(r,s)=(n,0)$, the equality clearly holds. When $(r,s)=(n-1,1)$, it was proved in \cite[4.5]{Pappas2000}. By Theorem \ref{thm:main-equation}, the affine chart \eqref{equ:embedding-of-chart} is reduced under our assumptions (note that $n=r+s\geq 2s+1$ since $r>s$). 
Hence the closed immersion is an isomorphism. The assertion then follows from Theorem \ref{thm:coherent}(3).
\end{proof}

\subsection{Moduli description for unitary Shimura varieties}\label{sec:moduliforSV}
We next provide a moduli description for the canonical model of unitary similitude Shimura varieties. For simplicity, we focus on the case where the hermitian form is defined over an imaginary quadratic field. Except possibly for the Hasse principle, all statements remain valid in the more general setting of a CM extension $\bfF/\bfF_0$, though the formulations become more involved. For a general treatment, see \cite{LMZ}.

Let $\bfF/\mbQ$ be an imaginary quadratic extension. Consider an $\bfF$-vector space $\bfV$ of dimension $n$, equipped with a skew-hermitian form $\varphi:\bfV\times\bfV\to\mbQ$. We define the unitary similitude group $\bfG=\GU(\bfV,\varphi)$. Fixing an embedding $\bfF\hookrightarrow \mbC$, we assume that the base change $\varphi_{\mbC}$ has signature $(r,s)$ ($r\neq s$). This determines a conjugacy class of homomorphisms $h:\Res_{\mbC/\mbR}\mbG_m\to \bfG_\mbR$, and the pair $(\bfG,h)$ defines a Shimura datum of PEL type.

Now fix a prime $p\neq 2$ that is ramified in $\bfF/\mbQ$ so that $\bfF\otimes \mbQ_p/\mbQ_p$ is a ramified quadratic extension.
Let $G:=\bfG_{\mbQ_p}$ be the unitary similitude group associated with the induced skew-hermitian form $\phi=\varphi_{\mbQ_p}$. 
Let $\Lambda\subset V:=\bfV_F$ be a self-dual lattice.
Let $K_p=\Stab_G(\Lambda)\subset G(\mbQ_p)$. This is a stabilizer quasi-parahoric subgroup. Fix a compact open subgroup $K^p\subset \bfG(\mbA_f^p)$ and define $K=K_pK^p$ as a compact open subgroup of $\bfG(\mbA_f)$.

We define a Deligne--Mumford stack $\mcA_K$ over $\Spec O_F$ by modifying the construction in \cite{RZ96}. 
To an $O_F$-algebra $R$, we associate the groupoid of tuples $(A,\iota,\ov{\lambda},\ov{\eta})$ satisfying the following conditions:
\begin{enumerate}
\item $A$ is an abelian scheme over $\Spec R$ equipped with an $O_F$-action:
$$
\iota:O_F\to \End A\otimes \mbZ_{(p)};
$$
\item $\ov{\lambda}$ is a $\mbZ_{(p)}$-homogeneous principal polarization of $A$, such that the induced Rosati involution is compatible with the conjugation of $O_F$;
\item $\ov{\eta}$ is a $K^p$-level structure
$$
\ov{\eta}:H_1(A,\mbA_f^p)\simeq \bfV\otimes \mbA_f^p\text{ mod }K^p,
$$
which respects the bilinear forms on both sides up to a constant in $(\mbA_f^p)^\times$;
\item The induced $O_F$-action on the Lie algebra $\Lie(A)$ satisfies the following conditions:
\begin{itemize}
\item(Kottwitz condition)
$$
\det(T\cdot \id-\iota(\pi)|\Lie(A))=(T-\pi)^r(T+\pi)^s;
$$
\item(Wedge condition)
$$
\bigwedge^{r+1}_R\bigl(
\iota(\pi)+\pi\mid \Lie(A)
\bigr)=0,\quad 
\bigwedge^{s+1}_R\bigl(
\iota(\pi)-\pi\mid \Lie(A)
\bigr)=0.
$$
\end{itemize}
\end{enumerate}

\begin{theorem}
Suppose $p>2$ and $p\geq 2s-1$.
The moduli functor $\mcA_K$ is represented by a flat and normal Deligne--Mumford stack over $\Spec(O_F)$. It defines an integral model of the Shimura variety $\Sh_K(\bfG,h)_{F}$.
\end{theorem}
\begin{proof}
The representability of $\mcA_K$ follows from \cite{RZ96} and \cite{Pappas2000}. Furthermore, we have the following local model diagram
\begin{equation*}
\begin{aligned}
\xymatrix{
&\wt{\mcA}_K\ar[dl]_-{p}\ar[dr]^-{q}&\\
\mcA_K&&\mbM_{(r,s)}^{\wedge}.
}
\end{aligned}
\end{equation*}
Here $p$ is a $\Stab(\Lambda)$-torsor, where $\Stab(\Lambda)$ is the quasi-parahoric group scheme associated with $K_p$. The map $q$ is smooth of relative dimension $\dim \Stab(\Lambda)$. The flatness then follows from Theorem \ref{thm:Pappas-rs}, and normality follows from Theorem \ref{thm:coherent}(1).
\end{proof}

\bibliographystyle{alpha}
\bibliography{main}
		
\end{document}